\documentclass[12pt, reqno]{amsart}
\usepackage{amsmath, amsthm, amscd, amsfonts, amssymb, graphicx, color}
\usepackage[bookmarksnumbered, colorlinks, plainpages]{hyperref}
\hypersetup{colorlinks=true,linkcolor=red, anchorcolor=green, citecolor=cyan, urlcolor=red, filecolor=magenta, pdftoolbar=true}

\newtheorem{theorem}{Theorem}[section]
\newtheorem{lemma}[theorem]{Lemma}

\newtheorem{corollary}[theorem]{Corollary}
\theoremstyle{definition}

\theoremstyle{remark}

\numberwithin{equation}{section}

\begin{document}
\setcounter{page}{1}

\title[Multiplication operators on Orlicz and weighted Orlicz spaces]{Multiplication operators on Orlicz and weighted Orlicz spaces}

\author[R. K. Giri, S. Pradhan]{Ratan Kumar Giri, Shesadev Pradhan}

\address{ Dept. of Mathematics, National Institute of Technology, Rourkela, Odisha, India.}
\email{\textcolor[rgb]{0.00,0.00,0.84}{giri90ratan@gmail.com,
pradhansh@nitrkl.ac.in}}




\date{\today}

\begin{abstract}
Let $(\Omega,\Sigma,\mu)$ be a $\sigma$-finite complete measure
space,  $\tau:\Omega\rightarrow\Omega$ be a measurable
transformation and $\phi$ be an Orlicz function. In this article,
first a necessary and sufficient condition for the bounded
multiplication operator $M_u$ on Orlicz space $L^\phi (\Omega)$
induced by measurable function $u$ to be completely continuous has
been established. Next by using Radon-Nikodym derivative $\omega =
\frac{d \mu\circ \tau^{-1}}{d \mu}$, the multiplication operator
$M_u$ on weighted Orlicz space $L^{\phi}_\omega(\Omega)$ have been
characterized.
\end{abstract} \maketitle

\section{Introduction and preliminaries}
Let $(\Omega,\Sigma,\mu)$ be a $\sigma$-finite complete measure
space, where $\Sigma$ is a $\sigma$-algebra of subsets of an
arbitrary set $\Omega$ and $\mu$ is a non-negative measure on
$\Sigma$. Let $\phi :[0,\infty)\rightarrow[0,\infty)$ be a
continuous convex function such that
 $\phi(x)=0$ if and only if $x=0$ with
 $\displaystyle {\lim_{x\rightarrow 0}\phi(x)/x=0}$ and
$\displaystyle {\lim_{x\rightarrow \infty}\phi(x)/x=\infty}.$
Such a function $\phi$ is known as an Orlicz function. Let
$L^{0}(\Omega)$ be denote the linear space of all equivalence
classes of $\Sigma$-measurable functions on $\Omega$, where we
identify any two functions are equal in the sense of $\mu$-almost
everywhere on $\Omega$. Then the functional $I_\phi :L^{0}(\Omega)
\rightarrow [0,\infty]$, defined by
$$I_\phi(f)=\int_\Omega\Phi(f(t))d\mu(t)$$ where $f\in
L^{0}(\Omega)$, is a pseudomodular \cite{Ames97}, which is also
defined as a modular in \cite{Ames98}. Let $L^\phi (\Omega)$ be the
set of all measurable function such that $\int_{\Omega}\phi(\alpha
|f|)d\mu < \infty$ for some $\alpha>0$. The space $L^{\phi}(\Omega)$
is called as Orlicz space and it is a Banach space with two norms:
the Luxemberg norm, defined as
$$ ||f||_\phi=\inf \left \{k>0:I_\phi\left(\frac{|f|}{k}\right)\leq 1\right\}$$ and the
Orlicz norm in the Amemiya form is given as
$$||f||^{0}_\phi = \displaystyle {\inf _{k>0}(1+I_\phi(kf))/k}.$$
 Note that the
equality of the Orlicz norm and the Amemiya norm was proved in
\cite{Ames82}. If $\phi(x)= x^p$, $1<p<\infty$, then
$L^{\phi}(\Omega)= L^p$, the well known Banach space of
$p$-integrable function on $\Omega$ with $||f||_\phi=
\left(\frac{1}{p}\right)^{\frac{1}{p}} ||f||_p$. It is well known
that $||f||_\phi \leq ||f||^{0}_\phi\leq 2||f||_\phi$ and
$||f||_\phi \leq 1$ if and only if $I_\phi(f)\leq 1$. Moreover, if
$A\in \Sigma$ and $0<\mu(A)<\infty$, then
$||\chi_A||_\phi=\frac{1}{\phi^{-1}(\frac{1}{\mu(A)})},$ where
$\chi_A$ is the characteristic function on $A$ and
$\phi^{-1}(t)=\inf \{ s>0 : \phi(s)>t\}$ is the right continuous
inverse of $\phi$.
\par An Orlicz function $\phi$ is said to be satisfied the $\Delta_2$-condition for all $x$ if there exists a positive constant $K$ such
that $\phi(2x)\leq K \phi(x)$ for all $x
> 0$. There is also an equivalent definition of the
above $\Delta_2$ condition i.e., the function $\phi$ is said to
satisfy $\Delta_2$-condition for all $x$ if and only if there is a
$K>0$ such that
$$\phi(Mx)\leq K\phi(x)$$ for any $M>1$ and for all $x>0$. With each Orlicz function $\phi$ we can associate another convex
function $\psi : [0,\infty)\rightarrow \mathbb{R}_+$ defined by
$\psi(y)= \sup \{ x|y|-\phi(x); x\geq 0\}$, $y\in [0,\infty)$ which
have the similar properties as $\phi$. The function $\psi$ is called
complementary function to $\phi$. If $\phi(x)= \frac{x^p}{p}$,
$1<p<\infty$, then the corresponding complementary function to
$\phi$ is $\psi(y)=\frac{y^q}{q}$ where $\frac{1}{p}+\frac{1}{q}=1$.
In general, simple functions are not necessarily dense in
$L^{\phi}(\Omega)$, but if $\phi$ satisfies $\Delta_2$-condition for
all $u>0$, then the simple function are dense in $L^{\phi}(\Omega)$.
Also if $\phi$ satisfies $\Delta_2$-condition for all $u>0$, then
for any $\sigma$-finite measure space the dual of the Orlicz space
$L^{\phi}(\Omega)$ is the Orlicz space $L^{\psi}(\Omega)$ generated
by complementary function $\psi$ to $\phi$. For more literature
concerning Orlicz spaces, we refer to Kufener, John and Fucik
\cite{Ames79}, Musielak \cite{Ames72}, Krasnoselskii and Rutickii
\cite{Ames80} and Rao \cite{Ames94}.
\par Let us recall that an
atom of a measurable space with respect to the measure $\mu$ is an
element $A \in \Sigma$ with $\mu(A)
> 0$ such that for each $F \in \Sigma$, if $F \subset A$ then either
$\mu(F ) = 0$ or $\mu(F ) = \mu(A)$.  The measure $\mu$ is said to
be purely atomic if every measurable set of positive measure
contains an atom and $\mu$ is nonatomic if there are no atoms for
$\mu$. In particular, a measure is nonatomic if for any $A\in\Sigma$
with $\mu(A)>0$, then there exists a measurable subset $B$ of $A$
such that $\mu(A)>\mu(B)>0$. If $A$ is nonatomic with $\mu(A)>0$
then one can construct a decreasing sequence of measurable sets
$A=A_1\supset A_2\supset A_3\supset\cdots $ such that
$$\mu(A)=\mu(A_1)>\mu(A_2)>\cdots >0\,\,.$$ It is known that every $\sigma$- finite measure space $(\Omega, \Sigma, \mu)$
can be decomposed into two disjoint sets $\Omega_1$ and $\Omega_2$,
such that $\mu$ is non-atomic over $\Omega_1$ and $\Omega_2$ is a
countable collection of disjoint atoms respectively (\cite{Ames78,Ames65}).
\par Let $u: \Omega\rightarrow \mathbb{C}$ be a measurable function. Then the rule taking $u$ to $u.f$, is a linear transformation on
$L^0(\Omega)$ and we denote the transformation by $M_u$. In the case
that $M_u$ is continuous, it is called multiplication operator
induced by $u$ \cite{mult2}.
\par The Multiplication operators received
considerable attention over the past several decades especially on
some measurable function spaces such as $L^P$- spaces, Begrman
spaces and a few cones of  Orlicz spaces, such that these operators
played an important role in the study of operators on Hilbert
Spaces. The basic properties of multiplication operators on
measurable function spaces are studied by more mathematicians. For
more of this operatros we refer to Takagi \& Yokouchi \cite{mult3},
Axler \cite{mult4} and Gupta \& Komal \cite{mult1}. The invertible,
compact and Fredholm multiplication operators on Orlicz spaces are
charectrized in the paper of Gupta \& Komal \cite{mult1}. Regarding
the compactness of the multiplication operators $M_u$ on Orlicz
spaces $L^\phi (\Omega)$, we have the following results
\cite{mult1}:
\begin{theorem}
 The operator $M_u$ is a compact operator if and only if $L^\phi(N(u,\epsilon)) $ is finite dimensional for each $\epsilon>0$, where
 $ N(u,\epsilon) = \{ x\in \Omega: |u(x)|\geq \epsilon\}$ and $ L^\phi(N(u,\epsilon))= \{f \in L^\phi (\Omega): f(x)=0 \,\, \forall x \notin N(u,\epsilon)\}$.
\end{theorem}
\begin{corollary}
 If $(\Omega,\Sigma,\mu)$ is a non-atomic measure space, then the only compact multiplication operator on $L^\phi (\Omega)$ are the zero operator.
\end{corollary}
\begin{corollary}
 If for each $\epsilon >0$, the set $N(u,\epsilon)$ contains only finitely many atoms, then $M_u$ is a compact operator on $L^\phi (\Omega)$.
\end{corollary}
In this paper, we establish some classic properties of
multiplication operators on Orlicz and weighted Orlicz spaces. In
section $2$, we show that a completely continuous multiplication
operator $M_u$ on Orlicz spaces,  is compact if and only if for each
$\epsilon >0$, the set $N(u,\epsilon)$ contains only finitely many
atoms. Then in section $3$, we give some necessary and sufficient
conditions for invertibility, boudedness, compactness of
multiplication operators between weighted Orlicz spaces.

\section{Completely continuous Multiplication operators on Orlicz spaces }
Let $X$ and $Y$ be two Banach spaces and $A: X \rightarrow Y$  be a
bounded linear operator. Recall that $A: X \rightarrow Y$ is said to
be completely continuous if $A$ maps every weakly compact set of $X$
into a compact set of $Y$. Equivalently, the operator $A$ is
completely continuous if for every sequence $\{x_n\}$ in $X$, we
have
$$ x_n\rightarrow x\,\,\, \mbox{weakly} \Rightarrow Ax_n \rightarrow Ax
\,\,\,\mbox{in norm.}$$ Now if the Orlicz function $\phi$ is
strictly increasing then its inverse $\phi^{-1}$ is uniquely
defined. We say that $\phi(x) \succ \succ x$ if
$\frac{\phi^{-1}(x)}{x}\rightarrow 0$ when $x\rightarrow \infty $.
The function  $x^p$, where $p>1$ is one example of such function
$\phi$. Now the following theorem characterize the completely
continuous property of multiplication operators on Orlicz spaces.
\begin{theorem}
Suppose that the Orlicz function $\phi$ satisfies $\Delta_2$
condition for all $x \geq 0$ and $\phi(x) \succ \succ x$. Also let
the multiplication operator $M_u: L^\phi (\Omega) \rightarrow L^\phi
(\Omega)$ be bounded. Then the following are equivalent:
\begin{enumerate}
\item The operator $M_u: L^\phi (\Omega)
\rightarrow L^\phi (\Omega)$ is compact.
\item The operator $M_u: L^\phi (\Omega)
\rightarrow L^\phi (\Omega)$ is completely continuous.
\item For any $\epsilon >0$, the set $N(u,\epsilon)= \{x\in \Omega:
|u(x)|\geq \epsilon \}$ contains finitely many atoms.
\end{enumerate}
\end{theorem}
\begin{proof}
$(1)\Rightarrow (2)$ : Follows from the definition.\\
$(3)\Rightarrow (1)$ : Follows from the previous corollary. Thus it
is enough to prove $ (2) \Rightarrow (3)$. Suppose that the result
is false that is, $\exists \, \epsilon_0 >0$  such that the set
$N(u, \epsilon_0)= E_0$ ( say ) does not consists finitely many
atoms. Then the set $E_0$ either contains a non-atomic subset or has
infinitely many atoms.\\
If $E_0$ contains a non-atomic subset, then by the definition of
non-atomic set, there exists a decreasing sequence $\{E_n\}$ such
that $E_n \subset E_0$ and $0< \mu(E_n) <\frac{1}{n}.$\\
Then for each $n$, define $$h_n:=
\phi^{-1}\left(\frac{1}{\mu(E_n)}\right)\chi_{E_n}.$$ Then $ h_n\in
L^\phi (\Omega)$ and $||h_n||_\phi = 1$. We shall show that the
sequence $\{h_n\}$ weakly converges to zero as $n\rightarrow
\infty$. Since simple functions are dense in $L^\psi (\Omega)$, it
is enough to show that for each measurable subset $F$ of $E_0$,
$\displaystyle{\int_\Omega h_n \chi_F d\mu \rightarrow 0}$ holds,
when $n\rightarrow \infty$.\\
Now since $\phi(x) \succ \succ x$ we have,
 \begin{eqnarray*}
 \left|\displaystyle{\int_\Omega h_n \chi_F d\mu}\right |  & = & \phi^{-1}\left(\frac{1}{\mu(E_n)}\right) \mu(F\cap
 E_n)\\
& \leq & \phi^{-1}\left(\frac{1}{\mu(E_n)}\right) \mu(E_n)\\
&= &
\frac{\phi^{-1}\left(\frac{1}{\mu(E_n)}\right)}{\frac{1}{\mu(E_n)}}
\rightarrow 0 \,\,\,\,\,\mbox{as\,\, $n\rightarrow \infty$}
 \end{eqnarray*}
Now assume that $E_0$ has infinitely may atoms and $\{E_n\}$ are
disjoints atoms in $E_0$. Now, if $\mu(E_n)\rightarrow 0$ as
$n\rightarrow \infty $, then by the similar arguments $\{h_n\}$
weakly converges to zero. Otherwise, $\mu(E_0)\geq \displaystyle
{\mu\left(\cup_{n=1}^\infty E_n\right)= \sum_{n=1}^\infty \mu(E_n)=
+ \infty} $ implies that $\mu(F\cap E_n) \rightarrow 0$ as
$n\rightarrow \infty$ for each $F \subseteq E_0$ with
$0<\mu(F)<\infty$. Thus we have,$$ \int_\Omega h_n \chi_F d\mu
\rightarrow 0$$ as $n\rightarrow \infty$. Hence $\{h_n\}$ weakly
converges to zero as $n\rightarrow \infty$ in both the cases.\\
 Now as the operator $M_u: L^\phi (\Omega)
\rightarrow L^\phi (\Omega)$ is completely continuous and $\{h_n\}$
weakly converges to zero, hence the sequence $\{M_u h_n\}$ converges
to $0$ in norm that is, $$ ||M_uh_n||_\phi \rightarrow
0\,\,\,\,\,\mbox{as}\,\,\, n\rightarrow \infty .$$
 Now we have,
\begin{eqnarray*}
||M_u h_n||_\phi & = & \displaystyle{\inf \left \{k>0:\int _\Omega
\phi
\left(\frac{|u. h_n|}{k}\right)d\mu \leq 1\right\}}\\
& \geq & \displaystyle{\inf \left \{k>0:\int _\Omega \phi
\left(\frac{|\epsilon_0. h_n|}{k}\right)d\mu \leq 1\right\}}\\
& = & \epsilon _0 \,\, \displaystyle{\inf \left \{\frac{k}{\epsilon
_0}>0:\int _\Omega \phi
\left(\frac{|h_n|}{\frac{k}{\epsilon_0}}\right)d\mu \leq
1\right\}}\\
& = & \epsilon _0 \,\,||h_n||_\phi\\
& = & \epsilon_0
\end{eqnarray*}
This implies that $||M_u h_n||_\phi \geq \epsilon_ 0$ for all $n$,
which contradiction to the fact that $\{M_u h_n\}$ converges to $0$
in norm. Therefore, we arrive at a contradiction and the proof is
complete.
\end{proof}

\section{Multiplication operators on weighted Orlicz spaces}
Let $\tau:\Omega\rightarrow\Omega$ be a measurable transformation,
that is, $\tau^{-1}(A)\in\Sigma$ for any $A\in\Sigma$. If
$\mu(\tau^{-1}(A))=0$ for any $A\in\Sigma$ with $\mu(A)=0$, then
$\tau$ is called as nonsingular. This condition implies that the
measure $\mu\circ\tau^{-1}$, defined by $\mu\circ\tau^{-1}(A)
:=\mu(\tau^{-1}(A))$ for $A\in\Sigma$, is
 absolutely continuous w.r.t $\mu$ ($\mu\circ\tau^{-1}\ll\mu$). Then
the Radon-Nikodym theorem implies that there exist a non-negative
locally integrable function $\omega$ on $\Omega$ such that
$$\mu\circ\tau^{-1}(A)=\int_{A}\omega(t)d\mu(t)\,\,\,\quad\mbox{for}
\,\, A\in\Sigma.$$ The weighted orlicz space is defined as
$$L^{\phi}_\omega(\Omega)=\left\{f \in L^{0}(\Omega):
\int_{\Omega}\phi(\alpha|f|)\omega(t)d\mu(t)<\infty,\,\,
\mbox{for\,\,some}\,\, \alpha >0\right\} $$ with the  norm $$
||f||_{\phi,\omega}=\inf \left\{k>0 : \int_\Omega
\phi\left(\frac{|f|}{k}\right)\omega(t) d\mu(t)\leq 1\right\},$$
where the weight $\omega$ is the Radon-Nikodym derivative of
$\mu\circ \tau^{-1}$ w.r.t $\mu$ ( page no 11, \cite{blum74}).
Define the set $$
 \tilde{L}^\phi_\omega(\Omega)= \left\{ f :\Omega\rightarrow \mathbb{R}
\,\,\,\mbox{ measurable}\,\, | \int_\Omega
\phi(|f(t)|)\omega(t)d\mu(t)<\infty \right\}.$$ Note that if the
Orlicz function $\phi\in\Delta_2$ , then
$\tilde{L}^\phi_\omega(\Omega)$ is a linear space and for each $f\in
L^\phi_\omega(\Omega)$, there is a $\alpha>0$ such that $$\alpha
f\in B_{\phi,\omega}= \left\{g\in\tilde{L}^\phi_\omega(\Omega):
\int_\Omega\phi(|g|)\omega d\mu\leq 1\right\}.$$ Also if
$\phi\in\Delta_2$, then $\tilde{L}^\phi_\omega(\Omega) =
L^\phi_\omega(\Omega)$ and also $||f||_{\phi,\omega}\leq 1$ if and
only if $I_{\phi,\omega}(f)\leq 1$. Proof of above two statement
follows almost similar lines in the case of  $L^\phi(\Omega)$ spaces
(\cite{Ames94}).
 Now we start with the following results.
\begin{theorem}
Let $L^\phi_\omega(\Omega)$ be weighted Orlicz space and $\phi$
satisfies $\Delta_2$ condition,  then for nonzero $f\in
L^\phi_\omega(\Omega)$, we have
$$\int_\Omega \phi\left(\frac{|f|}{||f||_{\phi , \omega}}\right)\omega(t)
d\mu(t)= 1.$$
\end{theorem}
\begin{proof}
Consider the map $\rho_f:\mathbb{R_+}\cup\{0\}\longrightarrow
\mathbb{R}$ defined by $$ \rho_f( k)= \int_\Omega \phi(k|f|)\omega
d\mu,\quad k \in \mathbb{R_+}\cup\{0\}.$$ Then $\rho_f$ is
nondecreasing with $\rho_f(0)=0$ and $\displaystyle {\lim_{k
\rightarrow \infty} \rho_f(k)=\infty}$. Now if $k_n\longrightarrow
k_0$, as $\phi$ is continuous hence, $ \phi(k_n |f|)\omega
\longrightarrow \phi (k_0|f|)\omega$. As $k_n\longrightarrow k_0$,
there exists $\alpha>0$ such that $|k_n|\leq\alpha$ for all $n$,
hence $\phi(k_n|f|)\omega\leq \phi(\alpha|f|)\omega$. Since $\phi$
satisfies $\Delta_2$ condition, hence $\tilde{L}^\phi_\omega(\Omega)
= L^\phi_\omega(\Omega)$. Thus we have,
$$ \alpha f\in \tilde{L}^\phi_\omega(\Omega) \Rightarrow
\int_\Omega\phi(\alpha |f|)\omega d\mu<\infty.$$ So, by dominated
convergence theorem, we get $$\int_\Omega\phi(k_n|f|)\omega d\mu
\longrightarrow \int_\Omega\phi(k_0|f|)\omega d\mu.$$ Hence $\rho_f$
is continuous, monotone increasing  function on $[0,\infty)$. By
intermediate value theorem, there exists  $k'$ such that
\begin{align*}
 & \rho_f(k')=1\\
 \Rightarrow & \int_\Omega \phi(k'|f|)\omega d\mu =1
\end{align*}
From the definition of $||f||_{\phi,\omega}$, we get
$k'=||f||_{\phi,\omega}^{-1}$. Hence the result follows.
\end{proof}
\begin{corollary}
For any $F\in\Sigma$ with $0<\mu(F)<\infty $, we have
$||\chi_F||_{\phi,\omega}=\frac{1}{\phi^{-1}\left(\frac{1}{\mu(\tau^{-1}(F))}\right)}$.
\end{corollary}
\begin{proof}
From the previous theorem, we have
\begin{align*}
& \int_\Omega
\phi\left(\frac{\chi_F}{||\chi_F||_{\phi,\omega}}\right)\omega d\mu
=1 \\
\Rightarrow &\, \int_\Omega
\phi\left(\frac{\chi_F}{||\chi_F||_{\phi,\omega}}\right)
d\mu\circ\tau^{-1} =1 \\
\Rightarrow &\, \phi\left(\frac{1}{||\chi_F||_{\phi,\omega}}\right)
\mu\circ\tau^{-1}(F)=1 \\
\Rightarrow &\, \phi\left(\frac{1}{||\chi_F||_{\phi,\omega}}\right)
=
\frac{1}{\mu\circ\tau^{-1}(F)}\\
\end{align*}
Hence, it follows that $||\chi_F||_{\phi,\omega} =
\frac{1}{\phi^{-1}\left(\frac{1}{\mu(\tau^{-1}(F))}\right)}$
\end{proof}
\begin{lemma}
Let $\mu\left(\tau^{-1}(E)\right) \geq \mu(E)$ for every $E\in
\Sigma$. Then convergence in norm implies convergence in measure
$\mu$ in the weighted Orlicz space $L^{\phi}_\omega(\Omega)$.
\end{lemma}
\begin{proof}
Let $\{f_n\}$ be a sequence in $L^\phi_\omega(\Omega)$ such that
$f_n\rightarrow f$ in norm i.e., $||f_n- f||_{\phi,
\omega}\rightarrow 0$ as $n\rightarrow \infty$. Then $\exists \, n_0
\in \mathbb{N}$ such that $||f_n- f||_{\phi, \omega}\leq 1$ for all
$n\geq n_0$. For a given $\epsilon >0$, let us consider the set $E=
\{ x\in \Omega : |f_n(x)-f(x)|\geq \epsilon\}$. Now we have,
\begin{eqnarray*}
\int_\Omega \phi(|f_n(x)-f(x)|)\omega d\mu & = & \int_E
\phi(|f_n(x)-f(x)|)\omega d\mu + \int_{E^c}
\phi(|f_n(x)-f(x)|)\omega d\mu\\
&\geq & \phi(\epsilon) \int_E \omega d\mu\\
& = & \phi(\epsilon) \mu(\mu(\tau^{-1}(E))\\
& \geq & \phi(\epsilon ) \mu(E)
\end{eqnarray*}
This implies that
\begin{eqnarray}
\mu(E) \leq \frac{1}{\phi(\epsilon)}\displaystyle{\int_\Omega
\phi(|f_n(x)-f(x)|)\omega d\mu}.\label{eq:3}
\end{eqnarray}
 Also for $n\geq n_0$,
\begin{eqnarray*}
\frac{1}{||f_n- f||_{\phi, \omega}} \int_\Omega
\phi(|f_n(x)-f(x)|)\omega d\mu  \leq \int_\Omega
\phi\left(\frac{|f_n(x)-f(x)|}{||f_n- f||_{\phi,
\omega}}\right)\omega d\mu \leq 1
\end{eqnarray*}
This shows that
\begin{eqnarray}
\displaystyle {\int_\Omega \phi(|f_n(x)-f(x)|)\omega d\mu \leq
||f_n- f||_{\phi, \omega}}\,\,\,\, \mbox{for all}\,\, n \geq
n_0.\label{eq:4}
\end{eqnarray}
 Thus from (\ref{eq:3}) and (\ref{eq:4}) we have, $$ \mu(E) \leq
\frac{1}{\phi(\epsilon)} ||f_n- f||_{\phi, \omega} \,\,\,\,\mbox{for
\,\,all} \,\,n \geq n_0.$$ This shows that $\mu(E)\rightarrow 0$ as
$n\rightarrow \infty$. Therefore, convergence in norm implies
convergence in measure also.
\end{proof}

\begin{theorem}
Assume that $\mu\left(\tau^{-1}(E)\right) \geq \mu(E)$ for every
$E\in \Sigma$. Then every linear transformation $M_u :
L^\phi_\omega(\Omega) \rightarrow L^\phi_\omega(\Omega) $ is always
bounded.
\end{theorem}
\begin{proof}
Let $\{f_n\}$ be a sequence in $L^\phi_\omega(\Omega)$ such that
$f_n\rightarrow f$ and $\{M_uf_n\}$ converges to some $g\in
L^\phi_\omega(\Omega)$. By the previous corollary, as the
convergence in norms implies convergence in measure hence $f_n
\rightarrow f$ in $\mu$. Therefore, we can find a subsequence
$\{f_{n'}\}$ of $\{f_n\}$ such that
\begin{eqnarray*}
\displaystyle{\lim_{n\rightarrow \infty } f_{n'}(x)}  =  f(x)
\,\,\,\mbox{a.e.}
\end{eqnarray*}
Therefore,
\begin{eqnarray}
 \displaystyle{ \lim_{n\rightarrow \infty }u(x)
f_{n'}(x)}  =  u(x)f(x) \,\,\,\mbox{a.e.}\label{eq:1}
\end{eqnarray}
Also, $uf_n \rightarrow g$ in measure and hence $uf_{n'} \rightarrow
g$ in measure also. Therefore we can find a subsequence
$\{f_{n''}\}$ of $\{f_{n'}\}$ such that
\begin{eqnarray}
\displaystyle{ \lim_{n\rightarrow \infty }u(x) f_{n''}(x)}= g
\,\,\mbox{a.e.}\label{eq:2}
\end{eqnarray}
Hence from (\ref{eq:1}) and (\ref{eq:2}) we can conclude that
$u\cdot f = g $ a.e. Thus by closed graph theorem, $M_u$ is bounded.
\end{proof}
\begin{theorem}
The multiplication operator $M_u$ is bounded on weighted Orlicz
spaces $L^\phi_\omega(\Omega)$ if and only if $u \in
L^{\infty}(\mu)$. Moreover, $||M_u||= ||u||_\infty$.
\end{theorem}
\begin{proof}
Let us suppose that $u\in L^{\infty}(\mu)$. Then we have,
$$\int_\Omega \phi \left(\frac{|M_u f|}{||u||_\infty \cdot
||f||_{\phi,\omega}} \right)\omega d\mu \leq \int_\Omega
\phi\left(\frac{|f|}{||f||_{\phi, \omega}} \right)\omega d\mu = 1$$
This shows that $$I_{\phi,\omega}\left(\frac{|M_u f|}{||u||_\infty
\cdot ||f||_{\phi,\omega}} \right) \leq 1$$ This implies that
$$||M_u(f)||_{\phi,\omega} \leq ||u||_\infty\cdot ||f||_{\phi,\omega}.$$
Conversely, let $M_u$ be a bounded operator on
$L^\phi_\omega(\Omega)$. To show $u\in L^{\infty}(\mu)$. Suppose
not, then for every $n\in\mathbb{N}$, the set $E_n = \{x\in \Omega
:|u(x)|>n\}$ has a positive measure with respect to $\mu$. As $\mu$
is $\sigma$-finite, hence we can assume that $\mu(E_n) <\infty$.
Take $f= \chi_{E_n}$. Then $f\in L^\phi_\omega(\Omega)$ and by the
result that $\displaystyle{\int_\Omega
\phi\left(\frac{|f|}{||f||_{\phi , \omega}}\right)\omega(t) d\mu(t)=
1}$ for all nonzero $f\in L^\phi_\omega(\Omega)$, we have
\begin{align*}
&
\displaystyle{\int_{E_n}\phi\left(\frac{n}{||u\cdot\chi_{E_n}||_{\phi,\omega}}
\right)\omega d\mu \leq \int_\Omega\phi\left(\frac{|u\cdot
\chi_{E_n}|}{||u\cdot\chi_{E_n}||_{\phi,\omega}} \right)\omega d\mu
=1 }\\
\Rightarrow & \,
\phi\left(\frac{n}{||u\cdot\chi_{E_n}||_{\phi,\omega}}
\right)\int_{E_n} d\mu\circ \tau^{-1} \leq 1\\
\Rightarrow &
\,\phi\left(\frac{n}{||u\cdot\chi_{E_n}||_{\phi,\omega}} \right)
\leq \frac{1}{\mu(\tau^{-1}(E_n))}\\
\Rightarrow & \, n\cdot
\frac{1}{\phi^{-1}\left(\frac{1}{\mu(\tau^{-1}(E_n))}\right)} \leq
||u\cdot\chi_{E_n}||_{\phi,\omega}\\
\Rightarrow & \, n\cdot ||\chi_{E_n}||_{\phi,\omega} \leq ||M_u
\chi_{E_n}||_{\phi,\omega}\\
\Rightarrow & \, ||M_u|| \geq n
\end{align*}
for all $n\in \mathbb{N}$. This contradicts to the boudedness of
$M_u$. Hence $u \in L^{\infty}(\mu)$.\\
Next to show $||M_u||_{\phi,\omega}= ||u||_\infty$. Let $\delta >0$,
then the set $ E= \{ x \in \Omega: |u(x)|\geq ||u||_\infty -\delta
\} $ has a positive measure. Then we have,
$$\displaystyle{\int_\Omega \phi\left( \frac{||u||_\infty -\delta}{||M_u\chi_E||_{\phi,\omega}}\chi_E \right)\omega d\mu
\leq \int_\Omega \phi\left( \frac{|u\chi_E|}{||M_u\chi_E||_{\phi,\omega}}\right)\omega d\mu = 1} $$
Hence $||\chi_E||_{\phi,\omega}\leq \frac{||M_u \chi_E||_{\phi,\omega}}{||u||_\infty- \delta}$. This implies that
$||M_u|| \geq ||u||_\infty - \delta$. Since this is true for every $\delta>0$, hence $||M_u||\geq ||u||_\infty$ and hence the result.
\end{proof}
\begin{theorem}
 Assume that $\mu(\Omega)<\infty$ and Orlicz function $\phi$ satisfy $\Delta_2$ for all $x$. Then the set of all multiplication operators on $L^\phi_\omega(\Omega)$ forms  a maximal abelian subalgebra of
 $B\left(L^\phi_\omega(\Omega) \right)$.
\end{theorem}
\begin{proof}
 Let $m = \{ M_u: u\in L^\infty\}$. Then obviously $m$ is an abelian subalgebra of $B\left(L^\phi_\omega(\Omega) \right)$. For showing $m$ is maximal it is enough to show that if the bounded operator $T$ commutes with every element of $m$, then $T\in m$. Let $e:\Omega \rightarrow \mathbb{C}$ be a function
 identically equal to unity. Then $e \in L^\phi_\omega(\Omega)$. Take $v=T(e)$. Then for $E\in \Sigma $ we have,
 $$T(\chi_E) = T(M_{\chi_E}(e))= M_{\chi_E}(T(e))=M_{\chi_E}(v)= \chi_E \cdot v = M_v(\chi_E).$$
 Now claim is that $v\in L^\infty$. If not then the set $F=\{ x\in \Omega: |v(x)|>k\}$ has positive measure for every $k\in \mathbb{N}$. Hence we have,
 $$\displaystyle{\int_\Omega \phi\left(\frac{|k\chi_F|}{||T(\chi_F)||_{\phi,\omega}} \right)\omega d\mu \leq
 \int_\Omega \phi\left(\frac{|v\chi_F|}{||T(\chi_F)||_{\phi,\omega}} \right)\omega d\mu =
 \int_\Omega \phi\left(\frac{|M_v(\chi_F)|}{||M_v(\chi_F)||_{\phi,\omega}} \right)\omega d\mu=1}$$
 This implies that
 $$||T(\chi_F)||_{\phi,\omega}\geq k ||\chi_F||_{\phi,\omega}$$ for all $k\in \mathbb{N}$, which contradicts to the boundedness of the operator $T$.
 Hence $v$ must be in $L^\infty$. Now as $T$ and $M_v$ are agree on simple functions and simple function are dense in $L^\phi_\omega(\Omega)$, hence we must have
 $T=M_v$. Therefore $m$ is an maximal abelian subalgebra of $ B\left(L^\phi_\omega(\Omega) \right)$.
\end{proof}
\begin{corollary}
 If $\mu(\Omega)<\infty$ and Orlicz function $\phi$ satisfy $\Delta_2$ for all $x$, then the multiplication operator $M_u$ is invertible if and only if $u$ is
 invertible in $L^\infty$.
\end{corollary}
\begin{proof}
 Let $T$ is the inverse of $M_u$. Then $TM_u= M_u T= I$, where $I$ is the identity operator. Again for $M_v\in m$ we have,
 $T(M_v(f)) = T(v\cdot f)= T(v\cdot u\cdot T(f)) = T(M_u(v\cdot T(f))) =v\cdot T(f)= M_v(T(f))$. This shows that $T$ commutes with every element of $m$. As $m$ is
 an maximal abelian subalgebra, hence $T=M_v$ for some $v\in L^\infty$ and it can be shown that $v$ is the inverse  of $u$ in $L^\infty$.\\
 The other part is obvious.
\end{proof}
\begin{theorem}
 Let $\mu(\Omega)<\infty$ and $M_u\in B\left(L^\phi_\omega(\Omega) \right)$. Then $M_u$ is compact operator if and only if
 $L^\phi_\omega\left(N(u,\epsilon)\right)$ is finite dimensional for every $\epsilon >0$, where $ L^\phi_\omega\left(N(u,\epsilon)\right)
 = \{ f\in L^\phi_\omega(\Omega): f(x)=0; \,\,\forall\,\, x \notin \,\, N(u,\epsilon)\}$.
\end{theorem}
\begin{proof}
 Suppose $M_u$ is a compact operator on $L^\phi_\omega(\Omega)$. Since $L^\phi_\omega\left(N(u,\epsilon)\right)$ is invariant space under the multiplication
 operator $M_u$, hence $M_u\vert_{L^\phi_\omega\left(N(u,\epsilon)\right)}:L^\phi_\omega\left(N(u,\epsilon)\right) \rightarrow L^\phi_\omega\left(N(u,\epsilon)\right)$ is a compact
 operator as it is a restriction of the compact operator $M_u$. Now on $N(u,\epsilon)$, $|u(x)| \geq \epsilon$, where $\epsilon >0$. So inverse of $u$ is exists and by the
 previous corollary the operator $M_u\vert_{L^\phi_\omega\left(N(u,\epsilon)\right)}$ is invertible. This implies that the map $M_u\vert_{L^\phi_\omega\left(N(u,\epsilon)\right)}$ is onto. So
 by open mapping theorem, the map $M_u\vert_{L^\phi_\omega\left(N(u,\epsilon)\right)}$ is open. Hence the image of the open unit ball $B_1\left(L^\phi_\omega\left(N(u,\epsilon)\right)\right)$ is open in
$ L^\phi_\omega\left(N(u,\epsilon)\right)$ under the map
$M_u\vert_{L^\phi_\omega\left(N(u,\epsilon)\right)}$. It is known
that for an infinite dimensional space $X$, there is no open subsets
of $X$ which has a compact closure. Thus if
$L^\phi_\omega\left(N(u,\epsilon)\right)$ is infinite dimensional,
then the image
$M_u\vert_{L^\phi_\omega\left(N(u,\epsilon)\right)}\left(B_1\left(L^\phi_\omega\left(N(u,\epsilon)\right)\right)
\right)$ has no compact closure, which contradicts to fact that
$M_u\vert_{L^\phi_\omega\left(N(u,\epsilon)\right)}$ is compact.
Therefore, the space $L^\phi_\omega\left(N(u,\epsilon)\right)$ must
be
finite dimensional for every $\epsilon >0$.\\
Conversely, suppose that
$L^\phi_\omega\left(N(u,\frac{1}{n})\right)$ is finite dimensional
for every $n\in \mathbb{N}$. As $M_u$ is bounded, hence $u\in
L^\infty$. Now for every $n\in \mathbb{N}$,  define $u_n : \Omega
\rightarrow \mathbb{R/C}$ by
$$u_n(x)=\begin{cases} u(x)\,;~~~
\mbox{if}~~ x \in u\left(\Omega, \frac{1}{n}\right)\\
0\,;~~~~~~~\,\,\,\,\,\,\,\,\mbox{otherwise}
\end{cases}$$
where $u\left(\Omega, \frac{1}{n}\right) = \{x \in \Omega: |u(x)| >\frac{1}{n}\}$. Then we have
\begin{eqnarray*}
 \int_\Omega \phi\left(\frac{n |(M_{u_n}-M_u)f|}{||f||_{\phi,\omega}} \right)\omega d\mu
 & = & \int_ {u\left(\Omega, \frac{1}{n}\right)^c}\phi\left(\frac{n |u\cdot f|}{||f||_{\phi,\omega}} \right)\omega d\mu \\
 & \leq & \int_ {u\left(\Omega, \frac{1}{n}\right)^c}\phi\left(\frac{|f|}{||f||_{\phi,\omega}} \right)\omega d\mu \\
 & \leq & \int_ {\Omega}\phi\left(\frac{|f|}{||f||_{\phi,\omega}} \right)\omega d\mu \leq 1
\end{eqnarray*}
This shows that $||(M_{u_n}-M_u)f||_{\phi,\omega} \leq \frac{1}{n} ||f||_{\phi,\omega}$ and hence $M_{u_n}\rightarrow M_u$ when $n \rightarrow \infty$.
 But range$(M_{u_n}) =  L^\phi_\omega\left(N(u,\frac{1}{n})\right)$, which is finite dimensional. Therefore, $M_{u_n}$ are all finite rank operator and hence
 compact operator for every $n\in\mathbb{N}$. Since $M_u$ is a limit of  a sequence of finite rank compact operators, hence $M_u$
 is compact.
\end{proof}

\begin{corollary}
If the measure $\mu$ is non-atomic, then the only compact
multiplication operator from $L^\phi_\omega$ into itself is the zero
operator.
\end{corollary}

\bibliographystyle{amsplain}

\end{document}